\documentclass[11pt,fullpage]{article}

\usepackage[pdftex]{graphicx}
\usepackage{latexsym}
\usepackage{color}
\usepackage{multirow}

\usepackage{amsmath}
\usepackage{amssymb}
\usepackage{amsthm}
\usepackage{amsfonts}
\usepackage{bm}

\newtheorem{theorem}{Theorem}
\newtheorem{lemma}{Lemma}

\newtheorem*{assumption}{Assumption}
\newtheorem{corollary}{Corollary}

\long\def\ignore#1{}

\newcommand{\be}{\begin{equation}}
\newcommand{\ee}{\end{equation}}
\newcommand{\beqn}{\begin{eqnarray}}
\newcommand{\eeqn}{\end{eqnarray}}

\newcommand{\bfm}[1]{\mbox{\boldmath $#1$}}
\newcommand{\bbeta}{\bfm{\beta}}
\newcommand{\btheta}{\bfm{\theta}}

\newcommand{\bbetas}{\widehat{\bbeta}_{Slope}}
\newcommand{\bthetas}{\widehat{\btheta}_{Slope}}

\newcommand{\bp}{{\bf p}}

\newcommand{\bY}{{\bf Y}}
\newcommand{\bx}{{\bf x}}
\newcommand{\bX}{{\bf X}}
\newcommand{\bu}{{\bf u}}
\newcommand{\bw}{{\bf w}}
\newcommand{\mE}{\mathbb{E}}
\newcommand{\cu}{{\cal U}}
\newcommand{\cl}{{\cal L}}
\newcommand{\cb}{{\cal B}}

\newcommand{\cE}{{\cal E}}
\newcommand{\cC}{{\cal C}}
\newcommand{\cS}{{\cal S}}
\newcommand{\cH}{{\cal H}}
\newcommand{\cX}{{\cal X}}
\newcommand{\M}{\mathfrak{M}}

\begin{document}

\title{\bf High-dimensional classification by sparse logistic regression}

\author{{\bf Felix Abramovich}\\
Department of Statistics\\
 and Operations Research \\
Tel Aviv University \\
Israel \\
{\it felix@post.tau.ac.il}
\and
{\bf Vadim Grinshtein} \\
Department of Mathematics\\ 
and Computer Science\\
The Open University of Israel\\ 
Israel\\
{\it vadimg@openu.ac.il}
}

\date{}

\maketitle

\begin{abstract}
We consider high-dimensional binary classification by sparse logistic regression. We propose a model/feature selection procedure based on penalized maximum likelihood with a complexity penalty on the model size and derive 
the non-asymptotic bounds for its misclassification excess risk. To assess its tightness we establish  the corresponding minimax lower bounds.
The bounds can be reduced under the additional low-noise condition. The proposed
complexity penalty is remarkably related to the VC-dimension of a set of sparse linear classifiers.
Implementation of any complexity penalty-based criterion, however, requires a combinatorial search over all possible models. To find a model selection procedure computationally feasible
for high-dimensional data, we extend the Slope estimator for logistic regression  
and show that under an additional weighted restricted eigenvalue condition it is
rate-optimal in the minimax sense.
  
\end{abstract}


\bigskip

\section{Introduction} \label{sec:intr}
Classification is one of the most important setups in statistical learning and has been
studied in various contexts.  Theoretical foundations of classification
are presented in the books of Devroye, Gy\"orfi and Lugosi (1996) and Vapnik (2000), while the
surveys of the state-of-the-art can be found in Boucheron, Bousquet and Lugosi (2005) and
Giraud (2015, Section 9).

Consider a general (binary) classification with a (high-dimensional) vector of features $\bx \in \mathbb{R}^d$ and the outcome class label $Y|\bx \sim Bin(1,p(\bx))$.
The accuracy of a classifier $\eta$
is defined by a misclassification error $R(\eta)=P(Y \neq \eta(\bx))$. It is well-known that $R(\eta)$
is minimized by the Bayes classifier $\eta^*(\bx)=I\{p(\bx) \geq 1/2\}$.
However, the probability function $p(\bx)$ is unknown and the resulting classifier $\widehat{\eta}(\bx)$ should be designed from the data $D$: a random sample of $n$ independent observations $(\bx_1,Y_1),\ldots, (\bx_n,Y_n)$. 
The design points $\bx_i$ may be considered as fixed or random.  
The corresponding (conditional) misclassification error of $\widehat{\eta}$ is $R(\widehat{\eta})=P(Y \neq \widehat{\eta}(\bx)|D)$ and the goodness of $\widehat{\eta}$ w.r.t. $\eta^*$ is measured by the misclassification
excess risk $\cE(\widehat{\eta},\eta^*)=ER(\widehat{\eta})-R(\eta^*)$.

A common general (nonparametric) approach for finding a classifier $\widehat{\eta}$ from the data is empirical risk minimization (ERM), where minimization of 
a true misclassification error $R(\eta)$  is replaced by minimization of the corresponding empirical risk
$\widehat{R}_n(\eta)=\frac{1}{n}\sum_{i=1}^n I\{Y_i \neq \eta(\bx_i)\}$ over a given class of classifiers. Misclassification excess risk of ERM classifiers  has been intensively studied in the literature (see, e.g., Boucheron, Bousquet and Lugosi, 2005 and
Giraud, 2015, Section 9 for surveys  and references therein).
However, ERM can
be hardly used directly in practice due to its computational cost and is typically relaxed by some related convex minimization surrogate (e.g., SVM). 

Another possibility to obtain $\widehat{\eta}$ is to estimate $p(\bx)$ from the data by some $\widehat{p}(\bx)$
and use a plug-in classifier of the form  $\widehat{\eta}(\bx)=I\{\widehat{p}(\bx) \geq 1/2\}$.
A standard approach is to assume some (parametric or nonparametric) model for $p(\bx)$. In this paper we consider one of the most commonly used models -- logistic regression, where it is assumed that
$p(\bx)=\frac{\exp(\bbeta^t \bx)}{1+\exp(\bbeta^t \bx)}$ and $\bbeta \in \mathbb{R}^d$ is a vector of unknown regression coefficients. The corresponding Bayes classifier 
is a linear classifier $\eta^*(\bx)=I\{p(\bx) \geq 1/2\}=I\{\bbeta^t \bx \geq 0\}$.
One then estimates $\bbeta$ from the data by the maximum likelihood estimator (MLE) $\widehat{\bbeta}$,
plugs-in $\widehat{\bbeta}$ (or, equivalently, $\widehat{p}(\bx)$) and
the resulting (linear) classifier is $\widehat{\eta}(\bx)=I\{\widehat{p}(\bx) \geq 1/2\}=I\{\widehat{\bbeta}^t \bx \geq 0\}$. Unlike ERM, the MLE  $\widehat{\bbeta}$ though not available in the closed form, can be nevertheless obtained numerically by the fast iteratively reweighted least squares
algorithm (McCullagh and Nelder, 1989, Section 2.5). Nonparamertric
plug-in classifiers were considered in Yang (1999), Koltchinskii and Beznosova (2005), Audibert and Tsybakov (2007).

In the era of ``Big Data'', however, the number of features $d$ describing the objects for classification might be very large and
even larger than the sample size $n$ (large $d$ small $n$ setups) that raises a severe
``curse of dimensionality'' problem. 
Reducing the dimensionality of a feature space by selecting a sparse subset of ``significant'' features becomes essential.
Thus, Bickel and Levina (2004), Fan and Fan (2008) showed that even in simple cases, high-dimensional 
classification without feature selection 
might be as bad as just pure guessing.

Nevertheless, unlike model selection in high-dimensional Gaussian regression that has been intensively studied in 2000s
(see Birg\'e and Massart, 2001, 2007; Abramovich and Grinshtein, 2010; Rigollet and Tsybakov, 2011; Verzelen, 2012 among many others), there are much less theoretical results on model/feature selection in classification. 
Devroy, Gy\"orfi and Lugosi (1996, Chapter 18) and Vapnik (2000, Chapter 4) considered selection from a sequence of classifiers within
a sequence of classes  by penalized ERM with the structural penalty depending
on the Vapnik-Chervonenkis (VC) dimension of a class. They established the oracle inequalities and the upper bounds for
the misclassification excess risk of the selected classifier but did not provide the lower
bound to assess its optimality. See also  Boucheron, Bousquet and Lugosi (2005, Section 8) 
for related penalized ERM approaches and references therein. Recall, however, that a
computational cost (even for a {\em given} model) is a serious drawback of any ERM-based procedure.

The main goal of the paper is to fill the gap. In particular, we investigate feature selection in sparse logistic regression classification. Although logistic regression is widely used in various classification problems, its rigorous theoretical ground has not
been yet properly established.
Model selection in a general framework of generalized linear models (GLM) and in logistic regression in particular was
studied in Abramovich and Grinshtein (2016). They proposed model selection procedure based on penalized maximum likelihood
with a complexity penalty on the model size and investigated the goodness-of-fit of the
resulting estimator in terms of the Kullback-Leibler risk. They derived the nonasymptotic bounds for this risk and showed that the resulting estimator is asymptotically minimax and adaptive to the unknown sparsity. In this paper we utilize their approach for classification and consider the corresponding plug-in classifier.
In particular, we show that the considered complexity penalty is
remarkably related to the VC-dimension of a set of sparse linear classifiers. 
We establish the non-asymptotic upper bound for misclassification excess risk of the resulting classifier and construct explicitly the design for which it is sharp in the minimax sense. We also
show that the excess risk bounds can be improved under the additional low-noise assumption.

Any model selection criterion based on a complexity penalty requires, however, a combinatorial search over all possible models that makes its usefulness problematic for high-dimensional data. A common remedy is to replace the original complexity penalty
by a related convex surrogate. The probably most well-known techniques is Lasso.
However, it can achieve only the sub-optimal rates under some
extra conditions on the design matrix $X$ (van de Geer, 2008). 
Recently, for Gaussian linear regression Bogdan {\em et al.} (2015) proposed a Slope estimator.
Bellec, Lecu\'e and Tsybakov (2018) showed that under certain additional conditions on $X$, Slope is rate-optimal 
for linear regression. We adapt it to the
logistic regression (and, in fact, to a general GLM) setup and extend the results of Bellec, Lecu\'e and Tsybakov (2018).

The rest of the paper is organized as follows. In Section \ref{sec:notation}
we present the model (feature) selection procedure for sparse logistic regression with fixed design
based on a general procedure of Abramovich and Grinshtein (2016) and provide the upper bounds for the resulting estimator in terms of Kullback-Leibler risk.
In Section \ref{sec:bounds} we apply it for classification to establish the non-asymptotic upper bound for its misclassification excess risk and derive the corresponding minimax
lower bounds. The improvement of the obtained risk bounds under the 
additional low-noise assumption is given in Section \ref{sec:lownoise}. In Section \ref{sec:slope} we consider the
logistic Slope classifier as a convex surrogate for the proposed feature selection procedure and show that its
misclassification excess risk is still rate-optimal under an extra weighted restricted eigenvalue condition on the design matrix $X$. 
The random design case is considered in Section \ref{sec:random} . Section \ref{sec:example} provides a short real-data example. All the proofs are given in the Appendix.

\section{Notation and preliminaries} \label{sec:notation}
Consider a sparse logistic regression model
\be \label{eq:model}
Y_i \sim Bin(1,p_i),\;\;\;\;\; \ln \frac{p_i}{1-p_i}=\bbeta^t \bx_i
\ee
with deterministic design points $\bx_i \in \mathbb{R}^d\;\;\;i=1,\ldots,n$,
where we assume that the unknown vector of regression coefficients $\bbeta \in \mathbb{R}^d$ is sparse. 

Let $d_0=||\bbeta||_0$ be the
size of true (unknown) model, where the $l_0$ (quasi)-norm of regression
coefficients $||\bbeta||_0$  is the number of nonzero entries.
Let $X \in \mathbb{R}^{n \times d}$ be the design matrix of rows $\bx_i$, $r=rank(X)$ and assume that
any $r$ columns of $X$ are linearly independent.

For the model (\ref{eq:model}) the log-likelihood function is
$$
\ell(\bbeta)=\sum_{i=1}^n \left\{\bbeta^t \bx_i~ Y_i-\ln(1+\exp(\bbeta^t \bx_i)\right\}
$$

Let $\M$ be the set of all $2^d$ possible models $M \subseteq
\{1,\ldots,d\}$. 
For a given model $M$, define $\cb_M =\{\bbeta \in \mathbb{R}^d: \beta_j=0\;{\rm if}\; j \not\in M\}$. The MLE $\widehat{\bbeta}_M$ of $\bbeta$ 
is then
\be \label{eq:mle}
\widehat{\bbeta}_M=\arg \max_{\widetilde{\bbeta} \in \cb_M}
\sum_{i=1}^n\left \{\widetilde{\bbeta}^t \bx_i Y_i -  \ln\left(1+\exp( \widetilde{\bbeta}^t \bx_i)\right)\right\}
\ee
The $\widehat{\bbeta}_M$ in (\ref{eq:mle})
is not available in the closed
form but can be obtained numerically by the iteratively reweighted
least squares algorithm (see McCullagh and Nelder, 1989, Section 2.5). The corresponding
MLE for probabilities $p_i$ are $\widehat{p}_{Mi}=\frac{\exp(\widehat{\bbeta}_M^t \bx_i)}
{1+\exp(\widehat{\bbeta}_M^t \bx_i)}$.

Select the model $\widehat{M}$ by the penalized maximum likelihood model selection criterion of the form 
\be \label{eq:binpen}
\widehat{M}=
\arg \min_{M \in \M} 
\left\{\sum_{i=1}^n \left( \ln\left(1+\exp(\widehat{\bbeta}_M^t \bx_i)\right)-
\widehat{\bbeta}_M^t \bx_i Y_i \right)  + Pen(|M|) \right\},
\ee
where $Pen(|M|)$ is a complexity penalty on the model size $|M|$. In fact, we can
restrict $\M$ in (\ref{eq:binpen}) to models with sizes at most $r$ since 
for any $\bbeta$ with $||\bbeta||_0 > r$, there necessarily exists another 
$\bbeta'$ with $||\bbeta'||_0 \leq r$ such that $X\bbeta=X\bbeta'$.

Within general GLM framework,
Abramovich and Grinshtein (2016) investigated  the goodness-of-fit of the resulting 
estimator  $\widehat{\bp}_{\widehat M}$. They considered the Kullback-Leibler divergence  $KL(\bp,\widehat{\bp}_{\widehat M})$ between the data distribution with the true probabilities $\bp=(p_1,\ldots,p_n)$ and its empirical distribution generated by $\widehat{\bp}_{\widehat M}$ given by
\be \label{eq:klbin}
KL({\bp},\widehat{\bp}_{\widehat M})=\frac{1}{n}\sum_{i=1}^n \left\{p_i \ln \left(\frac{p_i}{\widehat{p}_{{\widehat M}i}}\right)+
(1-p_i) \ln \left(\frac{1-p_i}{1-\widehat{p}_{{\widehat M}i}}\right)\right\},
\ee
and measured the accuracy of  $\widehat{\bp}_{\widehat M}$ by the corresponding
Kullback-Leibler risk $EKL(\bp,\widehat{\bp}_{\widehat M})$ (in fact, the Kullback-Leibler divergence $KL(\cdot,\cdot)$ in Abramovich and Grinshtein, 2016 was defined as $n$ times $KL(\cdot,\cdot)$ in this paper).

\begin{assumption}[\bf A] \label{as:A}
Assume that there exists $0 < \delta < 1/2$ such that $\delta < p_i < 1-\delta$  
or, equivalently, there exists $C_0>0$ such that $|\bbeta^t \bx_i| < C_0$ in (\ref{eq:model}) for all $i=1,\ldots,n$.
\end{assumption}
Assumption (A) prevents the variances $Var(Y_i)=p_i(1-p_i)$ to be
infinitely close to zero. 

Consider a set of models of size at most $d_0$, where $1 \leq d_0 \leq r$.  Obviously, $|M| \leq d_0$ iff  $||\bbeta||_0 \leq d_0$.
Abramovich and Grinshtein (2016) showed that for the complexity penalty 
\be \label{eq:penalty}
Pen(|M|)=c~ |M|\ln \frac{de}{|M|},\; |M|=1,\ldots,r-1\;\;\;{\rm and}\;\;\;Pen(r)=c \cdot r
\ee
in (\ref{eq:binpen}), where
$c >\frac{4}{\delta(1-\delta)}$, under Assumption (A), the upper bound of the
Kullnack-Leibler risk is given by
\be \label{eq:klupper}
\sup_{\bbeta: ||\bbeta||_0 \leq d_0} EKL({\bp},\widehat{\bp}_{\widehat M}) \leq C \frac{1}{\delta (1-\delta)}~ \frac{\min\left(d_0
\ln \frac{de}{d_0}, r\right)}{n}
\ee
for some $C>0$. They also derived the corresponding minimax lower bounds for the Kullback-Leibler risk and showed that for weakly-collinear design, the upper bound in (\ref{eq:klupper}) is of the optimal order (in the minimax sense).

The above results on the Kullback-Leibler risk can be  extended to model selection under additional structural constraints on the set of admissible models $\M$ (see Secition 4.1 of Abramovich and Grinshtein, 2016).

In what follows we utilize (\ref{eq:klupper}) to derive the upper bounds for the
misclassification excess risk of the corresponding plug-in classifier 
$\widehat{\eta}_{\widehat M}(\bx)=I\{\widehat{\bbeta}_{\widehat M}^t \bx \geq 0\}$.
 
To gain more insight into the complexity penalty (\ref{eq:penalty}) within
classification framework we present the following lemma on the Vapnik-Chervonenskis (VC) dimension of the set of all $d_0$-sparse linear classifiers~:
\begin{lemma} \label{lem:vc} Let 
$\cC(d_0)=\{\eta(\bx)=I\{\bbeta^t \bx \geq 0\}:
\bbeta \in \mathbb{R}^d,\;||\bbeta||_0 \leq d_0\}$ be 
the set of all $d_0$-sparse linear classifiers  and $V(\cC(d_0))$ its VC-dimension. Then,
\be \label{eq:vc}
d_0 \log_2 \left(\frac{2d}{d_0}\right) \leq V(\cC(d_0)) \leq 2 d_0 \log_2 \left(\frac{de}{d_0}\right)
\ee
\end{lemma}
Thus, the complexity penalty $Pen(|M|)$ in (\ref{eq:penalty}) is essentially proportional to the VC-dimension of the corresponding class of $|M|$-sparse linear classifiers $\cC(|M|)$.

%\section{Main results} \label{sec:main}
\section{Misclassification excess risk bounds} \label{sec:bounds}
In this section we apply the selected model $\widehat{M}$ in (\ref{eq:binpen}) for
classification and derive the bounds for the corresponding misclassification exceess risk. 

We consider first the fixed design. For a given design matrix $X$ the misclassification error of a classifier $\eta$ is
$R_X(\eta)=\frac{1}{n}\sum_{i=1}^n P(Y_i \neq \eta(\bx_i))$.
Following our previous arguments define a (linear) plug-in classifier 
\be \label{eq:classifier}
\widehat{\eta}_{\widehat M}(\bx)=I\{\widehat{\bbeta}_{\widehat M}^t \bx \geq 0\}
\ee
and consider its misclassification excess risk $\cE_X(\widehat{\eta}_{\widehat M},\eta^*)=E R_X(\widehat{\eta}_{\widehat M})-R_X(\eta^*)$,
where recall that the (ideal) Bayes classifier $\eta^*(\bx)=I\{\bbeta^t \bx \geq 0\}$ with
the true (unknown) $\bbeta$ in (\ref{eq:model}).
The remarkable results of Zhang (2004) and Bartlett, Jordan and McAuliffe (2006) establish
the relations between the Kullback-Leibler risk $EKL({\bp},\widehat{\bp}_{\widehat M})$ and the misclassification excess risk $\cE_X(\widehat{\eta}_{\widehat M},\eta^*)$:
\be \label{eq:bartlett}
\cE_X(\widehat{\eta}_{\widehat M},\eta^*) \leq \sqrt {2 EKL({\bp},\widehat{\bp}_{\widehat M})}
\ee
Thus, (\ref{eq:klupper}) and (\ref{eq:bartlett}) imply immediately the following upper bound for $\cE_X(\widehat{\eta}_{\widehat M},\eta^*)$:
\begin{theorem} \label{th:01upper}
Consider a sparse logistic regression model (\ref{eq:model}) with $||\bbeta||_0 \leq d_0$.
Let ${\widehat M}$ be a model selected in (\ref{eq:binpen}) with the complexity penalty (\ref{eq:penalty}) and consider the corresponding plug-in
classifier $\widehat{\eta}_{\widehat M}(\bx)$ in (\ref{eq:classifier}). Under Assumption (A), 
\be \label{eq:01upper}
\sup_{\eta^* \in \cC(d_0)}\cE_X(\widehat{\eta}_{\widehat M},\eta^*) \leq C_1 ~\sqrt{\frac{1}{\delta(1-\delta)}~\frac{\min\left(d_0
\ln \frac{de}{d_0}, r\right)}{n}}
\ee
for some $C_1>0$, simultaneously for all $1 \leq d_0 \leq r$.
\end{theorem}
The constant ${\sqrt \frac{1}{\delta (1-\delta)}}$ in 
(\ref{eq:01upper}) is a result of the direct application of (\ref{eq:klupper}) and (\ref{eq:bartlett}). It can be improved by establishing the similar relations between misclassification
excess risk and other losses rather than Kullback-Leibler in
(\ref{eq:bartlett}) and deriving the corresponding upper bounds for their risks. See, for example, the results and the proof of Theorem
\ref{th:01upper_random} below for the random design. In a way, the Kullback-Leibler loss implies the most conservative upper bound (Painsky and Wornell, 2018).

We now show that there exists a
design matrix $X_0$ for which the upper bound for the misclassification excess risk (\ref{eq:01upper}) is essential sharp (up to a probably different constant).

Consider the set of all possible $d_0$-sparse linear classifiers
$\cC(d_0)$ defined in Lemma \ref{lem:vc} and the case, where
a Bayes classifier $\eta^*(\bx)$ is not perfect, that is, $R(\eta^*) > 0$ (aka an {\em agnostic} model). Then, the
following result holds:
\begin{theorem} \label{th:01lower}
Consider a $d_0$-sparse agnostic logistic regression model (\ref{eq:model}), where $2 \leq d_0 \log_2\left(\frac{2d}{d_0}\right) \leq  n$.

Then, there exists a design matrix $X_0 \in \mathbb{R}^{n \times d}$ such that
\be
\label{eq:lower}
\inf_{\tilde  \eta} \sup_{\eta^* \in \cC(d_0)} \cE_{X_0}(\tilde{\eta},\eta^*) \geq
C_2 ~\sqrt{\frac{d_0 \ln \frac{de}{d_0}}{n}}
\ee
for some constant $C_2>0$, where the infimum is taken over all classifiers $\tilde{\eta}$ based on the data $(X_0,\bY)$.
\end{theorem}

%The condition (\ref{eq:vccond} ensures that $2 \leq V(\cC(d_0)) \leq n$
%(see Lemma \ref{lem:vc}).
Theorem \ref{th:01lower} is a particular case of Theorem \ref{th:lowerln} from
Section \ref{sec:lownoise} below.

The upper and lower bounds established in Theorem
\ref{th:01upper} and Theorem \ref{th:01lower} allow one to derive the asymptotic minimax
rate for misclassification excess risk in sparse logistic regression model as $n$ 
increases.  We allow the number of features $d$ to increase with $n$ as well and even faster
than $n$ ($d \gg n$ setup). The following immediate Corollary \ref{cor:minimax}  shows that  the proposed classifier $\widehat{\eta}_{\widehat M}$ is asymptotically minimax in terms of ``the worst case'' design and
adaptive to the unknown sparsity:

\begin{corollary} \label{cor:minimax}
Consider a $d_0$-sparse logistic regression agnostic model (\ref{eq:model}), where $d_0$ satisfies $2 \leq d_0 \log_2\left(\frac{2d}{d_0}\right) \leq n$.  Then, as $n$ and $d$ increase, for a fixed $\delta$ in  Assumption (A),
\begin{enumerate}
\item The asymptotic minimax misclassification excess risk  
$\sup_X \inf_{\tilde  \eta} \sup_{\eta^* \in \cC(d_0)} \cE_X(\tilde{\eta},\eta^*)$ is of the order 
$$
\sqrt{\frac{d_0 \ln\left(\frac{de}{d_0}\right)}{n}} 
\sim \sqrt{\frac{V(\cC(d_0))}{n}} 
$$ 
\item The classifier $\widehat{\eta}_{\widehat M}$ defined in (\ref{eq:classifier}), where
the model $\widehat{M}$ was selected by (\ref{eq:binpen}) with the complexity penalty
(\ref{eq:penalty}), attains the minimax rates 
simultaneously for all $2 \leq d_0 \log_2\left(\frac{2d}{d_0}\right) \leq n$. 
\end{enumerate}
\end{corollary}

Finally, we note that if the considered logistic regression model is misspecified and the Bayes
classifier $\eta^*$ is not linear, we still have the
following risk decomposition
\be \label{eq:decomp}
R_X(\widehat{\eta}_{\widehat M})-R_X(\eta^*)=\left(R_X(\widehat{\eta}_{\widehat M})-R_X(\eta^*_L)\right)+
\left(R_X(\eta^*_L)-R_X(\eta^*)\right),
\ee
where $\eta^*_L=\arg \min_{\eta \in \cC(d)} R_X(\eta)$ is the best (ideal) linear classifier.  
Our previous arguments can then be applied to the first term in the RHS of (\ref{eq:decomp})
representing the estimation error, while the second term is an approximation
error and measures the ability of linear classifiers to perform as good as $\eta^*$.
Enriching the class of classifiers may improve
the approximation error but will necessarily increase the estimation error in (\ref{eq:decomp}). 
In a way, it is similar to the variance/bias tradeoff in regression.

\section{Tighter risk bounds under low-noise condition} \label{sec:lownoise}
The main challenges for any classifier occur near the 
the boundary $\{\bx: p(\bx)=1/2\}$ (equivalently, a hyperplane $\bbeta^t \bx=0$ for the logistic regression model), where 
it is hard to predict the class label accurately. However, for regions, where $p(\bx)$ is bounded away from $1/2$ (margin or aka low-noise condition), the bounds for misclassification excess risk established in the previous
Section \ref{sec:bounds} can be improved. Following Massart and N\'ed\'elec (2006)  introduce the following low-noise assumption:
\begin{assumption}[\bf B] \label{as:B}   
Consider the logistic regression model (\ref{eq:model}) and assume that there exists $0 \leq h< 1/2$ such that
\be \label{eq:asb}
|p_i-1/2| \geq h \;\;{\rm or,\;equivalently,}\;\; 
|\bbeta^t \bx_i| \geq \ln\left(\frac{1+2h}{1-2h}\right)
\ee
for all $i=1,\ldots,n$.  
\end{assumption}
%For random design, somewhat more general low noise conditions were %considered in Mammen and Tsybakov (1999) and
%Tsybakov (2004).
Assumption (B) essentially assumes the existence of the ``corridor'' of width $2 \ln\left(\frac{1+2h}{1-2h}\right)$ that
separates the two sets $\{\bx_i:\bbeta^t \bx_i >0,\;i=1,\ldots,n\}$ and $\{\bx_i: \bbeta^t \bx_i <0,\;i=1,\ldots,n\}$.

For a given design matrix $X$, define $\cC_X(d_0,h)=\{\eta: \eta \in \cC(d_0),\;|\bbeta^t \bx_i| \geq \ln \left(\frac{1+2h}{1-2h}\right),\;i=1,\ldots,n\}$. Evidently, $\cC_X(d_0,0)=\cC(d_0)$ for any $X$.

Theorem \ref{th:upperln} below establishes the upper bound for
the misclassification excess risk of  the proposed classifier $\widehat{\eta}_{\widehat M}$ under the additional low noise
Assumption (B):

\begin{theorem} \label{th:upperln}
Consider a sparse logistic regression model (\ref{eq:model}), where $||\bbeta||_0 \leq d_0$. Assume that
there exist $0 < h < \Delta < 1/2$  such that 
\be \label{eq:hcond}
h \leq  |p_i-1/2| \leq \Delta
\ee
for all $i=1,\ldots,n$.

Let ${\widehat M}$ be a model selected in (\ref{eq:binpen}) with the complexity penalty (\ref{eq:penalty}) and consider the corresponding
classifier $\widehat{\eta}_{\widehat M}(\bx)$ in (\ref{eq:classifier}). Then, for all $ 1 \leq d_0 \leq r$,
\be \label{eq:upperln}
\sup_{\eta^* \in \cC_X(d_0,h)}\cE_X(\widehat{\eta}_{\widehat M},\eta^*) \leq C_1 \min \left(\sqrt{\frac{1-4h^2}{1-4\Delta^2}~\frac{\min\left(d_0
\ln \frac{de}{d_0}, r\right)}{n}}~,~\frac{1-4h^2}{1-4\Delta^2}~\frac{\min\left(d_0 \ln \frac{de}{d_0}, r\right)}{nh}\right)
\ee
for some $C_1>0$.
\end{theorem}

Thus, if the margin parameter $h$ is large enough, namely, $h >\sqrt{\frac{d_0 \ln \frac{de}{d_0}}{n}}$,  the misclassification excess risk bound (\ref{eq:01upper}) is reduced.  The classifier $\widehat{\eta}_{\widehat M}(\bx)$ does not depend on $h$ and the procedure is inherently adaptive to its value.
%Similar results for random design were obtained 
%in Massart and N\'ed\'elec (2006).

\ignore{For $\cC_X(d_0,h)$, the  minimax lower bound for the misclassification excess risk is given by the following Theorem \ref{th:lowerln}~:}

Similar to the previous Section \ref{sec:bounds}, one can construct
a design matrix for which the upper bound (\ref{eq:upperln}) is sharp: 
\begin{theorem} \label{th:lowerln}
Consider a $d_0$-sparse agnostic logistic regression model (\ref{eq:model}) with $2 \leq d_0 \log_2\frac{2d}{d_0} \leq  n$.

There exists a design matrix $X_0 \in \mathbb{R}^{n \times d}$ such that under Assumption (B)
\be \label{eq:lowerln}
\inf_{\tilde{\eta}} \sup_{\eta^* \in \cC_{X_0}(d_0,h)} \cE_{X_0}(\tilde{\eta},\eta^*) \geq C_2~ \min\left(\sqrt{\frac{d_0 \ln \frac{de}{d_0}}{n}}~,~\frac{d_0 \ln \frac{de}{d_0}}{nh}\right)
\ee
for some $C_2>0$.
\end{theorem}
The design matrix
$X_0$ is constructed explicitly in the proof of Theorem \ref{th:lowerln} in the Appendix.
Note that Theorem \ref{th:01lower} may be viewed as a particular case of
Theorem \ref{th:lowerln} for $h=0$.

\section{Logistic Slope classifier} \label{sec:slope}
Solving for $\widehat{M}$ in (\ref{eq:binpen}) requires generally a combinatorial search over all possible models in $\M$ that makes the use of complexity penalties to be computationally problematic when the number of features is large.
Greedy algorithms (e.g., forward selection) approximate the global solution of (\ref{eq:binpen}) by a stepwise sequence of local ones. However, they require 
strong constraints on the design matrix $X$ that can hardly hold for high-dimensional data. A more reasonable approach is convex relaxation, where the
original combinatorial problem is replaced by a related convex surrogate. Thus,
for linear-type complexity penalties of the form $Pen(|M|)=\lambda |M|=\lambda
||\bbeta||_0$, the celebrated Lasso replaces the $l_0$-(quasi) norm by $l_1$-norm:
%\be \label{eq:lassobin}
$$
\widehat{\bbeta}_{Lasso}=\arg \min_{\widetilde{\bbeta}} 
\left\{\sum_{i=1}^n \left( \ln\left(1+\exp(\widetilde{\bbeta}^t \bx_i)\right)-
\widetilde{\bbeta}^t \bx_i Y_i \right)  + \lambda ||\widetilde{\bbeta}||_1 \right\}
$$
%\ee
Assume that all the columns of the design matrix $X$ are normalized to have  unit norms. From the results of van de Geer (2008) it follows that under an assumption similar to Assumption (A) and certain
extra conditions on $X$, the logistic Lasso with a tuning parameter $\lambda$ 
of the order
$\sqrt{\ln d}$ results in sub-optimal Kullback-Leibler risk
$O\left(\frac{d_0}{n}\ln d \right)$ 
and, therefore, sub-optimal misclassification excess risk $O\left(\sqrt{\frac{d_0}{n}\ln d}\right)$. For Gaussian regression, Bellec, Lecu\'e and Tsybakov (2018) showed that under certain conditions on $X$, Lasso can achieve the optimal rate with {\em adaptively chosen} $\lambda$ by Lepski procedure.

Recently, 
for Gaussian regression, Bogdan {\em et al.} (2015) suggested the Slope estimator --  a penalized maximum likelihood estimator with a
{\em sorted} $l_1$-norm penalty defined as follows:
\be \label{eq:slopegauss}
\widehat{\bbeta}_{Slope}=\arg \min_{\widetilde{\bbeta}}\left\{||\bY-X\widetilde{\bbeta}||_2^2+
\sum_{j=1}^d \lambda_j |\widetilde{\bbeta}|_{(j)}\right\},
\ee
where $||\cdot||_2$ denotes the Euclidean norm in $\mathbb{R}^n$, 
$|\widetilde{\bbeta}|_{(1)} \geq \ldots \geq |\widetilde{\bbeta}|_{(d)}$ are the descendingly ordered absolute values of $\widetilde{\bbeta}_j$'s and $\lambda_1 \geq \ldots \geq \lambda_d > 0$ are the tuning parameters. It is a convex minimization problem.
Unlike the constant $\lambda$ in Lasso, there is a sequence of
decreasing $\lambda_j$ in Slope. 
Bellec, Lecu\'e and Tsybakov (2018)
proved that under a weighted restricted eigenvalue condition on $X$ with normalized columns, the
quadratic risk of the Slope estimator (\ref{eq:slopegauss}) with $\lambda_j=A \sqrt{\ln(2d/j)}$ for a certain constant $A$
is of the (rate-optimal) order $O\left(\frac{d_0}{n}\ln(\frac{de}{d_0})\right)$. 

We will now extend the above results for Slope for logistic regression and, in fact, for a general GLM (see the Appendix C). Naturally modifying the definition of the Slope estimator for the considered logistic regression model (\ref{eq:model}), define 
\be \label{eq:slopebin}
\widehat{\bbeta}_{Slope}=\arg \min_{\widetilde{\bbeta}}
\left\{\sum_{i=1}^n \left( \ln\left(1+\exp(\widetilde{\bbeta}^t \bx_i)\right)-
\widetilde{\bbeta}^t \bx_i Y_i \right) + \sum_{j=1}^d \lambda_j |\widetilde{\bbeta}|_{(j)}\right\},
\ee
where $\lambda_1 \geq \ldots \geq \lambda_d > 0$. 
Note that (\ref{eq:slopebin}) is also a convex program that makes the logistic Slope
estimator computationally feasible for high-dimensional data. The corresponding estimated probabilities
$\widehat{p}_{Slope,i}=\frac{\exp(\widehat{\bbeta}_{Slope}^t \bx_i)}{1+\exp(\widehat{\bbeta}_{Slope}^t \bx_i)},\;i=1,\ldots,n$.

As usual, any convex relaxation requires certain extra conditions on the restricted minimal eigenvalue of the design matrix $X$ over some set of vectors. 
In particular, similar to Gaussian regression considered in Bellec, Lecu\'e and Tsybakov (2018),
we assume the following {\em Weighted Restricted Eigenvalue} (WRE) condition for Slope estimator (\ref{eq:slopebin})~:
\begin{assumption}{(\bf $WRE(d_0,c_0)$ condition)}
Consider the sparse logistic regression model (\ref{eq:model}) with 
$||\bbeta||_0 \leq d_0$, where the columns of the design matrix $X$ are normalized to have unit norms.
Consider the set $\cS(d_0,c_0)=\{\bu \in \mathbb{R}^d: \sum_{j=1}^d \sqrt{\ln(2d/j)} |u|_{(j)} 
\leq (1+c_0) ||\bu||_2~ \sqrt{\sum_{j=1}^{d_0} \ln(2d/j)} \}$ and assume that
$X\bu \neq {\bf 0}$ for any $\bu \neq {\bf 0} \in \cS(d_0,c_0)$.
\end{assumption}
An interesting discussion on the relations between the WRE condition and the restricted eigenvalue condition (RE) required for Lasso is given in Section 8 of Bellec, Lecu\'e and Tsybakov (2018).

Define a restricted minimal eigenvalue $\kappa(d_0,c_0)$ as follows~:
$$
\kappa(d_0,c_0)=\min_{\bu \in \cS(d_0,c_0); \bu \neq {\bf 0}} \frac{||X\bu||_2}{||\bu||_2} > 0
$$

%Under the above additional condition $WRE(d_0,c_0)$ on $X$ we have 
\begin{theorem} \label{th:slope}
Consider a sparse logistic regression model (\ref{eq:model}), where $||\bbeta||_0 \leq d_0$, the columns of the design matrix $X$ are normalized to have unit norms and, in addition, $X$ satisfies the $WRE(d_0,c_0)$ condition for 
some $c_0>1$. Assume that Assumption (A) holds. 

Let the tuning parameters 
\be \label{eq:lambda}
\lambda_j= A~\frac{c_0+1}{c_0-1} \sqrt{\ln(2d/j)},\;\;\;j=1,\ldots,d
\ee
with the constant $A \geq 20 \sqrt{6}$.

Then,
\be  \label{eq:klslope}
\sup_{\bbeta: ||\bbeta||_0 \leq d_0} \mE KL({\bp},\widehat{\bp}_{Slope}) \leq 8A^2~
\frac{c_0^2}{(c_0-1)^2}~\frac{1}{\delta (1-\delta)}
\left(\frac{2\pi+8}{\ln(2d)}+\frac{1}{\kappa^2(d_0,c_0)}\right)
\frac{d_0}{n} \ln\left(\frac{2de}{d_0}\right)
\ee
for all $1 \leq d_0 \leq r$.

\ignore{
for the corresponding logistic Slope estimator (\ref{eq:slopebin}) with probability at least $1-\Delta$,
\be \label{eq:klslope}
KL({\bp},\widehat{\bp}_{Slope}) \leq C^*
\frac{1}{\delta (1-\delta)} \max\left\{\frac{(\sqrt{\pi/2}+\sqrt{2\ln \Delta^{-1}})^2}{d_0 \ln(2d/d_0)}~,~\frac{1}{\kappa^2(d_0,c_0)}\right\} \frac{d_0}{n} \ln\left(\frac{2de}{d_0}\right)
\ee
simultaneously for all $\bbeta \in \mathbb{R}^d$ such that $||\bbeta||_0 \leq d_0$, where $C^*=19200~ c_0^2/(c_0-1)^2$.
}
\end{theorem}
Note that $\lambda_j$'s in (\ref{eq:lambda}) are of the same form as those in Bellec, Lecu\'e
and Tsybakov (2018) for Gaussian regression but differ in a constant $A$. 

Theorem \ref{th:slope} is a particular case of 
Theorem \ref{th:slopeglm} for a general GLM (see Appendix C).

Using (\ref{eq:bartlett}) one immediately gets the corresponding result for the
misclassification exceess risk of the logistic Slope classifier:
\begin{corollary} \label{cor:slope}
Assume all the conditions of Theorem \ref{th:slope} and choose 
$\lambda_j$ according to (\ref{eq:lambda}). Consider the logistic Slope classifier 
$\widehat{\eta}_{Slope}(\bx)=I\{\widehat{\bbeta}^t_{Slope} \bx \geq 0\}$.
Then,
\be \label{eq:slopeclass}
\cE_X(\widehat{\eta}_{Slope},\eta^*) = O\left(\sqrt{\frac{d_0
\ln \frac{de}{d_0}}{n}}\right)
\ee
\end{corollary}
Thus, the logistic Slope estimator is computationally feasible and yet achieves the optimal
rates under the additional $WRE(d_0,c_0)$ condition on the design for all but
very dense models for which $d_0\ln(\frac{de}{d_0})>r$ (see Theorem  \ref{th:01upper}). 

Furthermore, following the arguments in the proof of Theorem \ref{th:upperln}, one can show that the bound (\ref{eq:slopeclass}) for $\cE_X(\widehat{\eta}_{Slope},\eta^*)$  may be reduced under
the additional low noise Assumption (B).

\ignore{
\begin{corollary} \label{cor:slopeln}
Assume all the conditions of Theorem \ref{th:slope} and, in addition,
the low noise Assumption (B). Then, for the logistic Slope classifier 
$\widehat{\eta}_{Slope}(\bx)$ with $\lambda_j$ in (\ref{eq:lambda}) we have
$$
\cE_X(\widehat{\eta}_{Slope},\eta^*) = O\left(\min\left\{\sqrt{\frac{d_0
\ln \frac{de}{d_0}}{n}}~,~\frac{d_0
\ln \frac{de}{d_0}}{nh}\right\}\right)
$$
\end{corollary}
}

\section{Random design} \label{sec:random}
The results above have been obtained for the fixed design. 
In machine learning, it is more common to consider classification with random design.
In this section we show that our main previous results for the fixed design can be extended for the random design.

Consider the following model:
\be \label{eq:model_random}
Y|(\bX=\bx) \sim B(1,p(\bx)), \;\;\;p(\bx)=E(Y|\bX=\bx)=\frac{\exp(\bbeta^t \bx)}{1+\exp(\bbeta^t \bx)}\;\;{\rm and}\;\;\bX \sim q(\bx),
\ee
where $q(\cdot)$ is a marginal density of $\bX$ with
a bounded support $\cX \subset \mathbb{R}^d$.
By re-scaling we can assume without loss of generality that $||\bx||_2 \leq 1$ for all $\bx \in \cX$, where recall that $||\cdot||_2$ is the Euclidean norm.

We assume that all $X_j$ are linearly independent. 
Hence, the minimal eigenvalue $\lambda_{\min}(G)$ of the matrix $G=E(\bX \bX^t)$ is strictly positive.
%We assume that its minimal eigenvalue $\lambda_{min}(G)$ is bounded
%away from zero,
%i.e. $\lambda_{\min}(G) \geq c_0$ for some positive $c_0$ not %depending on $d$.

Recall that the misclassification excess risk of a classifier $\widehat{\eta}$ designed from a random
sample $(\bX_1,Y_1),\ldots,(\bX_n,Y_n)$ from the joint distribution
$(\bX,Y)$ is $\cE(\widehat{\eta}_{\widehat M},\eta^*)=
ER(\widehat{\eta})-ER(\eta^*)$, where the Bayes classifier 
$\eta^*(\bx)=I(\bbeta^t \bx \geq 0)$. Adapting a general result on the minimax lower bound for the misclassification excess risk for random design (see, e.g., Devroye, Gy\"orfi and Lugosi, 1996, Chapter 14 and Boucheron, Bousquet and Lugosi, 2005, Section 5.5) for a $d_0$-sparse agnostic logistic regression model (\ref{eq:model_random}), by Lemma \ref{lem:vc} we have:
$$
\inf_{\tilde  \eta} \sup_{\eta^* \in \cC(d_0), q} \cE(\tilde{\eta},\eta^*) \geq C \sqrt{\frac{V(\cC(d_0))}{n}}
\geq
\tilde{C} ~\sqrt{\frac{d_0 \ln \frac{de}{d_0}}{n}}
$$

Similar to the fixed design setup, consider the penalized maximum likelihood model selection procedure (\ref{eq:binpen}) with the
complexity penalty 
\be \label{eq:penalty_random}
Pen(|M|)=C|M| \ln\frac{de}{|M|},\;\;\;|M|=1,\ldots,\min(d,n),
\ee 
where the exact choice for the constant $C$ will follow from the 
proof of Theorem \ref{th:01upper_randoml} below.

The following Assumption (A1) is a direct analog of Assumption (A) for random design: 
\begin{assumption}[\bf A1] \label{as:A1}
Assume that there exists $0 < \delta < 1/2$ such that $\delta < p(\bx) < 1-\delta$  
or, equivalently, there exists $C_0>0$ such that $|\bbeta^t \bx| < C_0$ in (\ref{eq:model_random}) for all $\bx \in \cX$.
\end{assumption}

The following Theorem \ref{th:01upper_random} extends the results of
Theorem \ref{th:01upper} for random design:
\begin{theorem} \label{th:01upper_random}
Consider a sparse logistic regression model (\ref{eq:model_random}), where $||\bbeta||_0 \leq d_0$. 
%Assume, in addition,  that $q$ is bounded
%away from zero, i.e. $q(\bx) \geq \underline{q} >0$ for all $\bx %\in S$.

Let ${\widehat M}$ be a model selected in (\ref{eq:binpen}) with the complexity penalty (\ref{eq:penalty_random}) and consider the corresponding plug-in
classifier $\widehat{\eta}_{\widehat M}(\bx)$ in (\ref{eq:classifier}). Under Assumption (A1), 
$$
\sup_{\eta^* \in \cC(d_0)}\cE(\widehat{\eta}_{\widehat M},\eta^*) \leq C \sqrt{\ln\left(\frac{1}{\delta \lambda_{\min}(G))}\right)~ \frac{d_0
\ln \frac{de}{d_0}}{n}}
$$
for some positive $C>0$ , simultaneously for all $1 \leq d_0 \leq \min(d,n)$.
\end{theorem}
Thus, the classifier $\widehat{\eta}_{\widehat M}(\bx)$ is adaptively
rate-optimal (in the minimax sense) for random design as well.

Theorem \ref{th:01upper_random} is a particular case of
Theorem \ref{th:01upper_randoml} stated below.
%See the proof for more details on $C(\delta,\lambda_{\min}(G))$.

We should note that similar upper bounds can be obtained for  
model selection by penalized ERM utilizing general results of 
Devroye, Gy\"orfi and Lugosi (1996, Chapter 18)
and Vapnik (2000, Chapter 4) on structural penalties depending on a 
VC-dimension and applying Lemma \ref{lem:vc} for their adaptation to sparse logistic regression. 
%The resulting upper bounds for the misclassification excess risk are %similar to those in (\ref{eq:01upper}) for the fixed design (up to an %extra $\ln n$-factor).
See also Boucheron, Bousquet and Lugosi, (2005, Section 8) for related ERM approaches and references therein.
Recall, however, that a computational cost is a crucial drawback of any
ERM-based procedure.

The misclassification excess risk
of $\widehat{\eta}_{\widehat M}(\bx)$ can be again improved under the low-noise condition which can even be formulated in a more general form for random design (Mammen and Tsybakov, 1999 and Tsybakov, 2004):
\begin{assumption}[\bf B1] \label{as:B1}
Assume that there exist $C>0$ and $\alpha \geq 0$  such that
\be \label{eq:low_noise.random}
P\left(|p(\bX)-1/2|\leq h\right) \leq C h^{\alpha}
\ee
%for all $0 < h < h*$.  
for all $0 < h < h^*$, where $h^* < 1/2$.  
\end{assumption}
Assumption (B) in
Section \ref{sec:lownoise}
for the fixed design can be viewed as a limiting case $\alpha=\infty$.

\begin{theorem} \label{th:01upper_randoml}
Consider a sparse logistic regression model (\ref{eq:model_random}),
where $||\bbeta||_0 \leq d_0$.
Let ${\widehat M}$ be a model selected in (\ref{eq:binpen}) with the complexity penalty (\ref{eq:penalty_random}) and consider the corresponding plug-in
classifier $\widehat{\eta}_{\widehat M}(\bx)$ in (\ref{eq:classifier}).

Under Assumptions (A1) and (B1), there exists $C>0$ such that
\be \label{eq:01upper_randoml}
\sup_{\eta^* \in \cC(d_0)}\cE(\widehat{\eta}_{\widehat M},\eta^*)
\leq C \left(\ln\left(\frac{1}{\delta \lambda_{\min}(G))}\right)~\frac{d_0
\ln \frac{de}{d_0}}{n}\right)^{\frac{\alpha+1}{\alpha+2}}
\ee
for all $1 \leq d_0 \leq \min(d,n)$.
\end{theorem}
Theorem \ref{th:01upper_random} (no low-noise assumption) corresponds to the extreme case
$\alpha=0$. For another extreme case
$\alpha=\infty$ (complete separation from 1/2),
the upper bound (\ref{eq:01upper_randoml}) is 
$O\left(\frac{d_0
\ln \frac{de}{d_0}}{n}\right)$ similar to the results of Section \ref{sec:lownoise}  for the fixed design.

The rates 
(\ref{eq:01upper_randoml}) can be reduced further under additional
conditions on the support $\cX$ and the density $q(\bx)$ using the arguments of Koltchinskii and Beznosova (2005) and Audibert and Tsybakov (2007) but this is beyond the scope of the paper. 

To extend the results of Theorem \ref{th:slope} for Slope estimator
for random design one needs the $WRE(d_0,c_0)$ condition to be held with high probability. It evidently depends on the marginal distribution $q(\bx)$. Thus, Bellec, Lecu\'e and Tsybakov (2018, Theorem 8.3) showed that it is satisfied for multivariate Gaussian and even sub-Gaussian
$\bX$ when $d_0\ln^2(de/d_0) \leq c n$ for some constant $c>0$, and under
mild conditions on the covariance matrix.

\section{Numerical example} \label{sec:example}
We now demonstrate the performance of the proposed feature selection and classification procedures on a short numerical real-data study.

The online marketing data contains various information about 500 customers registered at the site:  
their personal data (e.g., gender, country of living, etc.) and purchase activities during the past period (e.g., frequencies and types of purchases, purchase amounts, currencies, etc.). Overall, there were 61 explanatory variables. Given such data, one of the main goals is to predict customers who are about to become inactive in order to incentivize them to remain by various discounts.
The output then is a binary variable indicating
whether a customer was still active during the next time period or not.

The data was randomly split  into training ($n_1=400$) and
test ($n_2=100$) sets. The number of possible features $d=61$ was too large to perform a complete combinatorial search for penalized
maximum likelihood model selection procedure (\ref{eq:binpen}) with the
complexity penalty (\ref{eq:penalty}). Instead we used its forward
selection version, logistic Lasso and logistic
Slope classifiers. The corresponding tuning constants were chosen by 5-fold cross-validation and the resulting three classifiers were then applied to the test set. 

The best misclassification rate was achieved by Slope (17\%), followed by Lasso (19\%) and forward selection (22\%). 
In addition, we compared the sizes of the models selected by the
three classifiers. The conservative forward selection procedure yielded a very sparse model with only 2 predictors -- the time since last purchase and purchase amount.
On the other hand, Lasso with the CV-chosen tuning parameter is known to tend to select too many variables (see, e.g., 
B\"uhlmann and van de Geer, 2011, Section 2) and resulted in the model of size 13 by adding 11 other variables.  
The Slope classifier with a decreasing sequence of tuning parameters commonly implies even larger models (32 in the considered example). 
Note, however, that prediction and model identification are two different problems and, in particular, the choices of tuning parameters for them should be different.

\section*{Acknowledgments} 
The work was supported by the Israel Science Foundation (ISF), grants ISF-820/13 and ISF-589/18.
The authors would like to thank Noga Alon for his help in the proof of Lemma 1, Alexander Tsybakov for valuable remarks and Roi Granot
for the real-data example.

\section*{Appendix} \label{sec:appendix}
Throughout the proofs we use various generic positive constants, not necessarily
the same each time they are used even within a single equation.

\subsection*{Appendix A: Proof of Lemma \ref{lem:vc}}
%\begin{proof}
Denote for brevity $V=V(\cC(d_0))$.
For any fixed subset of $d_0$ $\beta_j$'s the VC of the corresponding set of
$d_0$-dimensional linear classifiers is known to be $d_0$ (e.g., Giraud, 2015,
Exercise 9.5.2). Then, by Sauer's lemma the maximal number of different labelling
of $V$ points in $\mathbb{R}^{d_0}$ that such set of classifiers can produce is  $\sum_{k=0}^{d_0} \binom{V}{k} \leq \left(\frac{Ve}{d_0}\right)^{d_0}$ (see, e.g., Giraud, 2015, Section 9.2.2). The overall number of different labelling is,
therefore, $\binom{d}{d_0} \sum_{k=0}^{d_0}
\binom{V}{k}$, and by the definition of $V(\cC(d_0))$ we have  
$$
2^V  \leq \binom{d}{d_0}~ \sum_{k=0}^{d_0}
\binom{V}{k} \leq \left(\frac{de}{d_0}\right)^{d_0} \left(\frac{Ve}{d_0}\right)^{d_0} \leq \left(\frac{de}{d_0}\right)^{2d_0}
$$
that implies an upper bound $V \leq 2~ d_0 \log_2\left(\frac{de}{d_0}\right)$.

On the other hand, take $k=\log_2(2d/d_0)$ and let $K$ be the $k \times 2^{k-1}$
matrix whose columns are all possible vectors with entries $\pm 1$ and the
first entry $1$. Note that $d_0 2^{k-1}=d$. Let $W$ be the $d_0k \times d$ 
block-wise matrix consisting of $d_0 \times d_0$ blocks, each being a $k \times
2^{k-1}$ matrix, where the diagonal matrices are copies of $K$, while all others
are zero matrices.  Thus, $W$ has $d_0 k=d_0\log_2(2d/d_0)$ rows. It is
easy to verify that these rows are shattered by  
half-spaces whose supporting
vectors $w$ have a single non-zero $\pm 1$ entry in each of the $d_0$ blocks and,
therefore, $V \geq d_0\log_2(2d/d_0)$.
%\end{proof}

\subsection*{Appendix B: Tighter bounds for low-noise condition}

\subsubsection*{B1: Proof of Theorem \ref{th:upperln}}
%\begin{proof}
Assumption (\ref{eq:hcond}) obviously implies Assumption (A) with $\delta=1/2-\Delta$. In addition, under (\ref{eq:hcond}), $Var(Y_i)=p_i(1-p_i)
\leq (1/2-h)(1/2+h)=(1-4h^2)/4$. Hence, adapting the results of 
Abramovich and Grinshtein (2016) on Kullback-Leibler risk in general GLM
framework for logistic regression, the upper bound (\ref{eq:klupper}) 
for $EKL({\bp},\widehat{\bp}_{\widehat M})$ can be 
improved:
\be \label{eq:klupperh}
\sup_{\bbeta: ||\bbeta||_0 \leq d_0} EKL({\bp},\widehat{\bp}_{\widehat M}) \leq C~ \frac{1-4h^2}{1-4\Delta^2}~ \frac{\min\left(d_0
\ln \frac{de}{d_0}, r\right)}{n}
\ee
and, therefore, from (\ref{eq:bartlett}) we have
$$
\sup_{\eta^* \in \cC_X(d_0,h)}\cE_X(\widehat{\eta}_{\widehat M},\eta^*) \leq C_1 ~\sqrt{\frac{1-4h^2}{1-4\Delta^2}~\frac{\min\left(d_0\ln \frac{de}{d_0}, r\right)}{n}}
$$

On the other hand, we can adapt the general Theorem 3 of Bartlett, Jordan and McAuliffe (2006)
for $\psi(f)=(1/2)\left((1-f)\ln(1-f)+(1+f)\ln(1+f)\right) \geq f^2/2$ 
corresponding to the Kullback-Leibler risk (Zhang, 2004, Section 3.5), $\alpha=1$ corresponding to (\ref{eq:hcond}) and $c=1/(2h)$ to get
$$ 
\cE_X(\widehat{\eta}_{\widehat M},\eta^*) \leq 
\frac{4}{h}~ EKL({\bp},\widehat{\bp}_{\widehat M})
$$
Applying (\ref{eq:klupperh}) implies then
$$
\cE_X(\widehat{\eta}_{\widehat M},\eta^*) \leq C_1
~\frac{1-4h^2}{1-4\Delta^2}~\frac{\min\left(d_0 \ln \frac{de}{d_0}, r\right)}{nh}
$$

%\end{proof}

\subsubsection*{B2: Proof of Theorem \ref{th:lowerln}}
%\begin{proof}

For any $\tilde{\eta}$ and $\eta^* \in \cC(d_0,h)$ we have
\be \label{eq:b11}
\cE_X(\tilde{\eta},\eta^*) = \frac{1}{n}\sum_{i=1}^n P(\tilde{\eta}_i \neq \eta^*_i) |2p_i-1| \geq \frac{2h}{n}~ E\left(\sum_{i=1}^n I\{\tilde{\eta}_i \neq
\eta^*_i\}\right)=\frac{2h}{n} E||\tilde{\eta}-\eta^*||_1
\ee
for any $X$.

As we have mentioned, the worst case scenario for classification is when $p_i=1/2 \pm h$ or, equivalently,
$|\bbeta^t \bx_i|=\ln\left(\frac{1+2h}{1-2h}\right)$.
Let $V=d_0 \log_2(2d/d_0)$.
In the proof of Lemma \ref{lem:vc} we constructed explicitly the matrix $W_{V \times d}$ whose rows $\bw_1,\ldots, \bw_V$ are shattered by $\cC(d_0)$. Then, for any $\bp=\{\frac{1}{2} \pm h\}^V$
there exists $\bbeta \in \mathbb{R}^d$ such that $||\bbeta||_0 \leq d_0$ and $\bbeta^t \bw_i = \ln \frac{p_i}{1-p_i}=
\pm \ln \frac{1+2h}{1-2h}$ for all $i=1,\ldots,V$. 
Define also the corresponding binary vector
${\bf b}$ with $b_i=I\{\bbeta^t \bw_i \geq 0\}$, that is, $b_i=1$ if $p_i=\frac{1}{2}+h$ and $b_i=0$ if
$p_i=\frac{1}{2}-h$. Obviously, the set of all ${\bf b}$'s is 
a hypercube $H^V=\{0,1\}^V$.

Define now a $n \times d$ design matrix $X_0$ with $\varkappa$ rows
of $\bw_1$, $\varkappa$ rows of $\bw_2$, ..., $\varkappa$ rows
of $\bw_{V-1}$ and the remaining $n-(V-1)\varkappa$ rows of $\bw_V$,
where an integer $1 \leq \varkappa \leq \lfloor\frac{n}{V-1}\rfloor$ will be defined later.
%Obviously, $rank(X_0)=V$. 

The proof will now follow the general scheme of the proof of Theorem 4
of Massart and N\'ed\'elec (2006) but with necessary modifications for the fixed design.

For any $\bp \in \{\frac{1}{2} \pm h\}^V$ and the corresponding ${\bf b} \in H^V$ define
an $n$-dimensional indicator vector $\eta_{\bf b}=(\underbrace{b_1,\ldots,b_1}_{\varkappa},\ldots,
\underbrace{b_{V-1},\ldots,b_{V-1}}_{\varkappa},
\underbrace{b_V,\ldots,b_V}_{n-(V-1)\varkappa})$ and let
$\tilde{C}_{X_0}(d_0,h)=\{\eta_{\bf b},\; {\bf b} \in H^V\}$.
By its design, $\tilde{\cC}_{X_0}(d_0,h) \subseteq
\{\eta: \eta \in \cC(d_0),\;|\bbeta^t \bx_{0i}|=\ln\frac{1+2h}{1-2h},\;i=1,\ldots,n\} \subseteq C_{X_0}(d_0,h)$.

Hence, we can reduce the minimax risk over the entire $\cC_{X_0}(d_0,h)$
to $\tilde{\cC}_{X_0}(d_0,h)$:
\be \label{eq:minmaxreduced}
\inf_{\tilde \eta}\sup_{\eta^* \in \cC_{X_0}(d_0,h)}
\cE_{X_0}(\tilde{\eta},\eta^*) \geq \inf_{\tilde \eta}\sup_{\eta^* \in \tilde{\cC}_{X_0}(d_0,h)}
\cE_{X_0}(\tilde{\eta},\eta^*)
\ee

Furthermore,
for a given $\tilde{\eta}$, define $\tilde{\eta}^*=\arg \min_{\eta
\in \widetilde{\cC}(d_0,h)} ||\tilde{\eta}-\eta||_1$. Then, for
any $\eta^* \in \tilde{\cC}_{X_0}(d_0,h)$ we have
\be \label{eq:l1proj}
||\tilde{\eta}^*-\eta^*||_1 \leq  ||\tilde{\eta}^*-\tilde{\eta}||_1 + 
||\tilde{\eta}-\eta^*||_1 \leq 2 ||\tilde{\eta}-\eta^*||_1
\ee
and, therefore, from (\ref{eq:b11})-(\ref{eq:l1proj})
\be \label{eq:b12}
\begin{split}
\inf_{\tilde \eta}\sup_{\eta^* \in \cC_{X_0}(d_0,h)}\cE_{X_0}(\tilde{\eta},\eta^*) & \geq \frac{h}{n}~ \inf_{\tilde{\eta}^* \in \tilde{\cC}_{X_0}(d_0,h)}\sup_{\eta^* \in \tilde{\cC}_{X_0}(d_0,h)}E ||\tilde{\eta}^*-\eta^*||_1 \\
& \geq \frac{h}{n}~ \varkappa
\inf_{\tilde{\bf b} \in H^V}\sup_{{\bf b}^* \in H^V}
E\left(\sum_{i=1}^{V-1} I\{\tilde{b}_i \neq b^*_i\}\right),
\end{split}
\ee
where  $\tilde{\bf b},{\bf b}^* \in H^V$ are the binary vectors
corresponding to $\tilde{\eta}^*$ and $\eta^*$ respectively (see above).

By a simple calculus one can verify that the square Hellinger distance $H^2\left(Bin(1,\frac{1}{2}+h),Bin(1,\frac{1}{2}-h)\right)$ between two Bernoulli distributions
$Bin(1,\frac{1}{2}+h)$ and $Bin(1,\frac{1}{2}-h)$ is
$1-\sqrt{1-4h^2}$. 
For any ${\bf b} \in H^V$ and the corresponding $\eta_{\bf b}$
define $\bp_{\bf b} \in \mathbb{R}^n$ as follows: $p_{{\bf b}i}=\frac{1}{2}+h$ if $\eta_{{\bf b}i}=1$ and $p_{{\bf b}i}=\frac{1}{2}-h$ if $\eta_{{\bf b}i}=0,\;i=1,\ldots,\varkappa (V-1)$, and $p_{{\bf b}i}=0,\;
i=\varkappa (V-1)+1,\ldots,n$.
Then, for any ${\bf b}_1, {\bf b}_2 \in H^V$ and the corresponding $\bp_{{\bf b}_1}$ and $\bp_{{\bf b}_2}$ we have
$$
H^2(\bp_{{\bf b}_1},\bp_{{\bf b}_2})=\frac{1}{n}\sum_{i=1}^{n} H^2\left(Bin(1,p_{b_{1i}}),Bin(1,p_{b_{2i}})\right)=\frac{\varkappa}{n} (1-\sqrt{1-4h^2}) \sum_{i=1}^{V-1} I\{b_{1i} \neq b_{2i}\}
$$

Hence, applying the version of Assouad's lemma given in Lemma 7 of
Barron, Birg\'e and Massart (1999) yields
$$
\inf_{\tilde{\bf b} \in H^V}\sup_{{\bf b}^* \in H^V}
E\left(\sum_{i=1}^{V-1} I\{\tilde{b}_i \neq b^*_i\}\right) \geq
\frac{V-1}{2} \left(1-\sqrt{2 \varkappa (1-\sqrt{1-4h^2}}\right) 
\geq \frac{V-1}{2} \left(1-\sqrt{8 \varkappa h^2}\right)
$$
that together with (\ref{eq:b12}) implies
\be \label{eq:infsup}
\inf_{\tilde \eta}\sup_{\eta^* \in \cC_{X_0}(d_0,h)}\cE_{X_0}(\tilde{\eta},\eta^*) \geq \varkappa \frac{h}{n} \frac{V-1}{2} \left(1-\sqrt{8 \varkappa h^2}\right)
\ee

\noindent
Consider two cases.
\newline
{\em Case 1.} $h \leq \frac{1}{6}$.
\noindent
\newline
For $h \geq \sqrt{\frac{V-1}{18n}}$, apply (\ref{eq:infsup}) for $\varkappa=\lfloor\frac{1}{18h^2}\rfloor$ (note that $2 \leq \varkappa \leq \lfloor\frac{n}{V-1}\rfloor$), to get
$$
\inf_{\tilde \eta}\sup_{\eta^* \in \cC_{X_0}(d_0,h)}\cE_{X_0}(\tilde{\eta},\eta^*) \geq \frac{V-1}{216 n h} \geq C_2 
\frac{d_0 \ln(\frac{de}{d_0})}{nh}
$$

For  $h < \sqrt{\frac{V-1}{18n}}$, one can follow all the above
arguments for $\tilde{h}=\sqrt{\frac{V-1}{18n}}$ and the corresponding
$\varkappa=\lfloor\frac{n}{V-1}\rfloor$ to have
$$
\inf_{\tilde \eta}\sup_{\eta^* \in \cC_{X_0}(d_0,h)}\cE_{X_0}(\tilde{\eta},\eta^*) \geq \inf_{\tilde \eta}\sup_{\eta^* \in \cC_{X_0}(d_0,\tilde{h})}\cE_{X_0}(\tilde{\eta},\eta^*)
\geq C_2 \sqrt{\frac{d_0 \ln(\frac{de}{d_0})}{n}}
$$

\noindent
\newline
{\em Case 2.} $h > \frac{1}{6}$.
\noindent
\newline
Set $\varkappa=1$ and note that 
$C_{X_0}(d_0,\frac{1}{2}) \subseteq C_{X_0}(d_0,h)$ for any $0 \leq h \leq \frac{1}{2}$. Hence, (\ref{eq:b12}) implies
%\be \label{eq:b13}
\be \nonumber
\begin{split}
\inf_{\tilde \eta}\sup_{\eta^* \in \cC_{X_0}(d_0,h)}\cE_{X_0}(\tilde{\eta},\eta^*) \geq  &
\inf_{\tilde \eta}\sup_{\eta^* \in \cC_{X_0}(d_0,\frac{1}{2})}\cE_{X_0}(\tilde{\eta},\eta^*) 
\geq \frac{1}{2n} \inf_{\tilde{\bf b} \in H^V}\sup_{{\bf b}^* \in H^V}
E\left(\sum_{i=1}^{V-1} I\{\tilde{b}_i \neq b^*_i\}\right) \\ 
\geq & \frac{1}{2n} \inf_{\tilde{\bf b} \in H^V}
\frac{1}{2^V} \sum_{{\bf b}_j \in H^V}E\left(\sum_{i=1}^{V-1} I\{\tilde{b}_i \neq b_{ji}\}\right)  \\
=& \frac{1}{2n} \inf_{\tilde{\bf b} \in H^V}\sum_{i=1}^{V-1}\frac{1}{2^V} \sum_{j=1}^{2^V} 
P(\tilde{b}_i\neq b_{ji})
\end{split}
\ee
By obvious combinatoric calculus, for any (binary) vector $\tilde{\bf b}$, 
$\frac{1}{2^V} \sum_{j=1}^{2^V} 
P(\tilde{b}_i\neq b_{ji})=\frac{1}{2}$ for any $i$ and, therefore,
$$
\inf_{\tilde \eta}\sup_{\eta^* \in \cC_{X_0}(d_0,h)}\cE_{X_0}(\tilde{\eta},\eta^*) \geq  \frac{V-1}{4n} 
\geq   C_2 \frac{d_0 \ln\frac{de}{d_0}}{nh}
$$
for large $h>\frac{1}{6}$ (in fact, larger than
any fixed $h_0$).

\subsection*{Appendix C: Slope estimator for a general GLM} 
Consider a GLM setup with a response variable $Y$ and a set of $d$
predictors $x_1,...,x_d$. We observe a series of independent
observations $(\bx_i,Y_i),\;i=1,\ldots, n$, where the design points $\bx_i \in
\mathbb{R}^p$ are deterministic. The distribution $f_{\theta_i}(y)$ of $Y_i$
belongs to a (one-parameter) natural exponential family with a natural
parameter $\theta_i$ and a scaling parameter $a$:
\be
\label{eq:modelglm} f_{\theta_i}(y)=\exp\left\{\frac{y \theta_i -
b(\theta_i)}{a}+c(y,a)\right\}
\ee
The function $b(\cdot)$ is assumed to be twice-differentiable.
In this case
$\mE(Y_i)=b'(\theta_i)$ and $Var(Y_i)=ab''(\theta_i)$. To complete
GLM we assume the canonical link
$\theta_i=\bbeta^{t}\bx_i$ or, equivalently, in the matrix form,
$\btheta=X\bbeta$, where $X_{n \times p}$ is the design matrix and
$\bbeta \in \mathbb{R}^p$ is a vector of the unknown regression
coefficients. The logistic regression (\ref{eq:model}) is a particular case of
a general GLM (\ref{eq:modelglm}) for the Bernoulli distribution $Bin(1,p_i)$, where the natural
parameter is $\theta=\ln\frac{p}{1-p}, b(\theta_0)=\ln(1+e^{\theta})$ and $a=1$.

Following  Abramovich and Grinshtein (2016) assume the extended version of
Assumption (A) for GLM~:
\begin{assumption}[\bf A'] \label{as:A'}
\noindent
\begin{enumerate}
\item Assume that $\theta_i \in \Theta$, where the parameter
space $\Theta \subseteq \mathbb{R}$ is a closed (finite or infinite) interval.
\item Assume that there exist constants $0 < \cl \leq \cu < \infty$ such that the
function $b''(\cdot)$ satisfies the following conditions:
\begin{enumerate}
\item $\sup_{t \in \mathbb{R}} b''(t) \leq \cu$
\item $\inf_{t \in \Theta} b''(t) \geq \cl$
\end{enumerate}
\end{enumerate}
\end{assumption}
Conditions on $b''(\cdot)$ in
Assumption (A') are intended to exclude two degenerate
cases, where the variance $Var(Y)$ is infinitely large or small. They
also ensure strong convexity of $b(\cdot)$ over
$\Theta$.  For the binomial distribution, $\cu=1/4$ and Assumption (A') reduces to Assumption (A) with $\cl=\delta(1-\delta)$.

Recall that the Slope estimator is a penalized maximum likelihood with an ordered
$l_1$-norm penalty and, therefore, defined for a GLM as follows:
\be \label{eq:slopeglm}
\bbetas= \arg \min_{\widetilde \bbeta}\left\{-\ell({\widetilde \bbeta})+\sum_{j=1}^d \lambda_j |{\widetilde \bbeta}|_{(j)} \right\}=
\arg \min_{\widetilde \bbeta} 
\left\{b(X{\widetilde \bbeta})^t {\bf 1}-\bY^t X{\widetilde \bbeta}+
\sum_{j=1}^d \lambda_j |{\widetilde \bbeta}|_{(j)} \right\}
\ee
for $\lambda_1 \geq \cdots \geq \lambda_d > 0$.
The corresponding Kullback-Leibler risk
\be \label{eq:mE}
\mE KL(\btheta,\bthetas)=
\frac{1}{n} \frac{1}{a} \left(b'(\btheta)^t(\btheta-\mE(\bthetas))-(b(\btheta)-\mE b(\bthetas))^t {\bf 1}\right)
\ee
where $\btheta=X\bbeta$ and $\bthetas=X\bbetas$ (see Abramovich and Grinshtein, 2016).

\begin{theorem} \label{th:slopeglm}
Consider a GLM (\ref{eq:modelglm}), where $||\bbeta||_0 \leq d_0$, the columns of the design matrix $X$ are normalized to have unit norms and $X$ satisfies the $WRE(d_0,c_0)$ condition for 
some $c_0>1$. Assume that Assumption (A') holds. 

Let 
\be \label{eq:lambdaglm}
\lambda_j= A~  \frac{c_0+1}{c_0-1}~\sqrt{\frac{\cu}{a}}~\sqrt{\ln(2d/j)},\;\;\;j=1,\ldots,d,
\ee
in (\ref{eq:slopeglm}) with the constant $A \geq 40\sqrt{6}$.

Then, simultaneously for all $\bbeta \in \mathbb{R}^d$ such that $||\bbeta||_0 \leq d_0$,
\begin{enumerate}
\item 
\be \label{eq:klslopeglm}
\begin{split}
P&\left(KL(\btheta,\bthetas)  \leq \frac{8A^2}{n} \frac{c_0^2}{(c_0-1)^2}
~ \frac{\cu}{\cl} \max\left\{\left(\sqrt{\pi/2}+\sqrt{2\ln \Delta^{-1}}\right)^2~,~\frac{d_0}{\kappa^2(d_0,c_0)} \ln\left(\frac{2de}{d_0}\right) \right\} \right) \\ 
&\geq 1-\Delta
\end{split}
\ee
for any $0 < \Delta < 1$.

\item 
\be  \label{eq:kl1slopeglm} 
\mE KL(\btheta,\bthetas) \leq 8A^2 \frac{c_0^2}{(c_0-1)^2}~
\frac{\cu}{\cl} \left(\frac{2\pi+8}{\ln(2d)}+\frac{1}{\kappa^2(d_0,c_0)}\right)
\frac{d_0}{n} \ln\left(\frac{2de}{d_0}\right)
\ee
\end{enumerate}
\end{theorem}

\begin{proof}
Since $\bbetas$ is the minimizer of (\ref{eq:slopeglm}), 
$$
-\ell(\bbetas)+\sum_{j=1}^d \lambda_j |\bbetas|_{(j)} \leq -\ell(\bbeta)+\sum_{j=1}^d \lambda_j |\bbeta|_{(j)}
$$
From (29) of Abramovich and Grinshtein (2016) one has
$$
n~ KL(\btheta,\bthetas)=\ell(\bbeta)-\ell(\bbetas)+
\frac{1}{a}(\bY-b'(\btheta))^t(\bthetas-\btheta)
$$
(recall that the Kullback-Leibler divergence $KL(\cdot,\cdot)$ in Abramovich and Grinshtein, 2016 was defined as $n$ times 
$KL(\cdot,\cdot)$ in this paper). Thus, 
\be \label{eq:a1}
KL(\btheta,\bthetas) \leq \frac{1}{n~ a}(\bY-b'(\btheta))^t(\bthetas-\btheta) +
\frac{1}{n}\left(\sum_{j=1}^d \lambda_j |\bbeta|_{(j)}-\sum_{j=1}^d \lambda_j |\bbetas|_{(j)} \right)
\ee
Let $\bu=\bbetas-\bbeta$.
Applying Lemma A.1 of Bellec, Lecu\'e and Tsybakov (2018) with $\tau=0$ implies 
\be \label{eq:B}
\sum_{j=1}^d \lambda_j |\bbeta|_{(j)}-\sum_{j=1}^d \lambda_j |\bbetas|_{(j)} 
\leq \sqrt{\sum_{j=1}^{d_0} \lambda_j^2}~ ||\bu||_2-\sum_{j=d_0+1}^d \lambda_j |u|_{(j)}
\ee
Consider now the first term of the RHS in (\ref{eq:a1}).  
Since the distribution of $Y$ belongs to the exponential family with
the bounded variance $ab''(\theta) \leq a\cu$ (Assumption (A')), a centered 
zero mean random variable $Y-b'(\theta)$ is sub-Gaussian with the scale factor
$\sqrt{a\cu}$, that is, $\mE e^{t(Y-b'(\theta))} \leq e^{a\cu t^2/2}$ and, therefore, $\mE e^{(Y-b'(\theta))^2/(6 \cu a)} \leq e$. 
Applying Theorem 9.1 of Bellec, Lecu\'e and Tsybakov (2018) (adapted to our normalization conditions on the columns of $X$) yields
\be \label{eq:subgauss}
\frac{1}{n a}(\bY-b'(\btheta))^t(\bthetas-\btheta) 
\leq \frac{40 \sqrt{6 \cu}}{n\sqrt{a}}~ \max\left(\sum_{j=1}^d |u|_{(j)} \sqrt{\ln(2d/j)}~,~
||\bthetas-\btheta||_2(\sqrt{\pi/2}+\sqrt{2 \ln \Delta^{-1}})\right)
\ee 
with probability at least
$1-\Delta$.

Set
\be \label{eq:H}
H(\bu)=\sum_{j=1}^d |u|_{(j)} \sqrt{\ln(2d/j)} \leq ||\bu||_2
\sqrt{\sum_{j=1}^{d_0} \ln(2d/j)}+\sum_{j=d_0+1}^d |u|_{(j)} \sqrt{\ln(2d/j)}
=\tilde{H}(\bu)
\ee
and
\be \label{eq:G}
G(\bu)= ||\bthetas-\btheta||_2 \left(\sqrt{\pi/2}+\sqrt{2 \ln \Delta^{-1}}\right)
\ee

The proof will now go along the lines of the proof of Theorem 6.1 of Bellec, Lecu\'e and Tsybakov (2018) for Gaussian regression with necessary adaptations to GLM and different normalization conditions on the columns of $X$.

To prove (\ref{eq:klslopeglm}) consider two cases.

\vspace{.2cm}
\noindent
$Case~  1.~ \tilde{H}(\bu) \leq G(\bu)$. In this case 
$$
||u||_2 \leq \frac{||\bthetas-\btheta||_2(\sqrt{\pi/2}+\sqrt{2 \ln \Delta^{-1})}} {\sqrt{\sum_{j=1}^{d_0} \ln(2d/j)}}
$$
and, therefore, combining (\ref{eq:lambdaglm}) and (\ref{eq:a1})-(\ref{eq:G}) with probability at least $1-\Delta$ yields
\be 
\begin{split} \label{eq:a2}
KL(\btheta,\bthetas) & \leq \frac{1}{n}~ A \sqrt{\frac{\cu}{a}}~ \frac{2c_0}{c_0-1}~ ||\bthetas-\btheta||_2(\sqrt{\pi/2}+\sqrt{2 \ln \Delta^{-1}}) \\
& \leq
\frac{1}{2n}~\left(\frac{A^2 \cu}{\epsilon a}~\left(\frac{2c_0}{c_0-1}\right)^2
(\sqrt{\pi/2}+\sqrt{2 \ln \Delta^{-1}})^2 + \epsilon ||\bthetas-\btheta||_2^2\right)
\end{split}
\ee
for any $\epsilon>0$. 

Lemma 1 of Abramovich and Grinshtein (2016) established the equivalence of
the Kullback-Leibler divergence $KL(\btheta,\bthetas)$ and the squared quadratic
norm $||\bthetas-\btheta||^2$ under Assumption (A'):
\be \label{eq:normeq}
\frac{\cl}{2a}||\bthetas-\btheta||_2^2 \leq n KL(\btheta,\bthetas) \leq \frac{\cu}{2a}||\bthetas-\btheta||_2^2
\ee
Hence, taking $\epsilon=\cl/(2a)$ in (\ref{eq:a2}) after a straightforward calculus yields 
\be \label{eq:case1}
KL(\btheta,\bthetas) \leq \frac{8}{n} \frac{c_0^2}{(c_0-1)^2}~ \frac{\cu}{\cl}~
A^2 (\sqrt{\pi/2}+\sqrt{2 \ln \Delta^{-1}})^2 
\ee 
with probability at least $1-\Delta$.

\vspace{.2cm}
\noindent
$Case~  2.~ \tilde{H}(\bu) > G(\bu)$. Using the definition of $\lambda_j$'s in (\ref{eq:lambdaglm}) and (\ref{eq:a1})-(\ref{eq:G}), with probability at least $1-\Delta$  we have
\be \label{eq:a3}
\begin{split}
KL(\btheta,\bthetas) & \leq \frac{1}{n}~40\sqrt{\frac{6\cu}{a}}\left(||\bu||_2~ 
\sqrt{\sum_{j=1}^{d_0} \ln(2d/j)}+\sum_{j=d_0+1}^d |u|_{(j)} \sqrt{\ln(2d/j)}\right) \\
& +\frac{1}{n}\left(\sqrt{\sum_{j=1}^{d_0} \lambda_j^2}~ ||\bu||_2-\sum_{j=d_0+1}^d \lambda_j |u|_{(j)}\right) \\
& \leq \frac{1}{n}\left(\frac{2c_0}{c_0+1} ||\bu||_2~\sqrt{\sum_{j=1}^{d_0} \lambda_j^2} -
\frac{2}{c_0+1} \sum_{j=d_0+1}^d \lambda_j |u|_{(j)} \right)
\end{split}
\ee
The $KL(\btheta,\bthetas) \geq 0$ and, therefore, the RHS of (\ref{eq:a3}) is necessarily positive. Thus,
$$
\sum_{j=1}^d |u|_{(j)} \sqrt{\ln(2d/j)} \leq ||\bu||_2~ \sqrt{\sum_{j=1}^{d_0}
\ln(2d/j)}+\sum_{j=d_0+1}^d|u|_{(j)} \sqrt{\ln(2d/j)}
\leq (1+c_0)||\bu|| \sqrt{\sum_{j=1}^{d_0} \ln(2d/j)}
$$
and, therefore, by $WRE(d_0,c_0)$ condition, (\ref{eq:a3}) implies
\be\nonumber
\begin{split}
KL(\btheta,\bthetas) & \leq \frac{1}{n}~ \frac{2 c_0}{c_0+1}~||\bu||_2~\sqrt{\sum_{j=1}^{d_0} \lambda_j^2}
~\leq~ \frac{1}{n}~ \frac{2 c_0}{c_0+1} \frac{||\bthetas-\btheta||_2}{\kappa(c_0,d_0)}
\sqrt{\sum_{j=1}^{d_0} \lambda_j^2} \\
& \leq \frac{1}{n} \left(\frac{c_0^2}{\epsilon(c_0+1)^2}~
\frac{\sum_{j=1}^{d_0} \lambda_j^2}{\kappa^2(c_0,d_0)}+\epsilon ||\bthetas-\btheta||_2^2\right)
\end{split}
\ee
for any $\epsilon>0$. 
Taking  $\epsilon=\cl/(4a)$ and exploiting the equivalence
between $KL(\btheta,\bthetas)$ and $||\bthetas-\btheta||_2^2$ in (\ref{eq:normeq})
imply that with probability at least $1-\Delta$,
$$
KL(\btheta,\bthetas) \leq \frac{1}{n}~ \frac{1}{\cl}~\frac{8 a  c_0^2}{(c_0+1)^2}  \frac{\sum_{j=1}^{d_0}\lambda_j^2}{\kappa^2(c_0,d_0)}
\leq \frac{8}{n} \frac{c_0^2}{(c_0-1)^2} \frac{\cu}{\cl}~ A^2 \frac{d_0 \ln(2de/d_0)}{\kappa^2(c_0,d_0)},
$$
where we used the definition (\ref{eq:lambdaglm}) of $\lambda_j$'s and  the upper bound
$\sum_{j=1}^{d_0}\ln(2d/j) \leq  d_0 \ln(2ed/d_0)$
(see, e.g., (2.7) of Bellec, Lecu\'e and Tsybakov, 2018).

To prove the second statement (\ref{eq:kl1slopeglm}) of the theorem denote
$C^*=8A^2 \frac{c_0^2}{(c_0-1)^2}$ and note that
\be \label{eq:a7}
\begin{split}
& C^*
\frac{1}{n}~ \frac{\cu}{\cl} \max\left\{\left(\sqrt{\pi/2}+\sqrt{2\ln \Delta^{-1}}\right)^2,~\frac{d_0}{\kappa^2(d_0,c_0)} \ln\left(\frac{2de}{d_0}\right) \right\} \\  
& \leq~ C^*
\frac{1}{n}~
\frac{\cu}{\cl} \max\left\{2\pi,~8 \ln \Delta^{-1},~\frac{d_0}{\kappa^2(d_0,c_0)} \ln\left(\frac{2de}{d_0}\right) \right\} \\
& \leq~C^*
\frac{1}{n}~\frac{\cu}{\cl} \max\left\{\max\left(\frac{2\pi}{\ln(2d)},~\frac{1}{\kappa^2(d_0,c_0)}\right)
d_0 \ln\left(\frac{2de}{d_0}\right),~8 \ln\Delta^{-1} \right\} \\
& \leq ~C^*
\frac{1}{n}~\frac{\cu}{\cl} \max\left\{\left(\frac{2\pi}{\ln(2d)}+\frac{1}{\kappa^2(d_0,c_0)}\right)
d_0 \ln\left(\frac{2de}{d_0}\right),~8 \ln\Delta^{-1} \right\}
\end{split}
\ee
Then, by integrating, (\ref{eq:klslopeglm}) and (\ref{eq:a7}) after a straightforward calculus yield
\be \nonumber
\begin{split}
\mE KL(\btheta,\bthetas)& =\int_0^{\infty} P\left(KL(\btheta,\bthetas) \geq t \right) 
dt  \\ 
& \leq   C^*
\frac{1}{n}~
\frac{\cu}{\cl} \left(\left(\frac{2\pi}{\ln(2d)}+\frac{1}{\kappa^2(d_0,c_0)}\right)
d_0 \ln\left(\frac{2de}{d_0}\right) + 8\left(\frac{2de}{d_0}\right)^{-\frac{d_0}{8}\max\{\frac{2\pi}{\ln(2d)},\kappa^{-2}(d_0,c_0)\}} \right) \\
& \leq  
C^*~
\frac{\cu}{\cl} \left(\frac{2\pi+8}{\ln(2d)}+\frac{1}{\kappa^2(d_0,c_0)}\right)
\frac{d_0}{n} \ln\left(\frac{2de}{d_0}\right)
\end{split}
\ee

\end{proof}

\subsection*{Appendix D: Proof of Theorem \ref{th:01upper_randoml}}
We first introduce several notations. Let $||g||_{L_2}=(\int_\cX g^2(\bx)d\bx)^{1/2}$ be a standard $L_2$-norm of a function $g$ and $||g||_{L_2(q)}=(\int_\cX g^2(\bx)q(\bx)d\bx)^{1/2}$ be the $L_2$-norm of $g$ weighted by the marginal distribution $q$ of $\bX$. In addition, the $L_\infty$-norm  $||g||_\infty=\sup_{\bx \in \cX} |g(\bx)|$. 

Applying Theorem 3 of Bartlett, Jordan and McAlliffe (2006) for
$\psi(g)=g^2/2$, Assumption (B1) implies that there exists $C>0$ such that
\be \label{eq:L2q}
\cE(\widehat{\eta}_{\widehat M},\eta^*) \leq C \left(E||\widehat{p}_{\widehat M}-p||^2_{L^2(q)}\right)^{\frac{\alpha+1}{\alpha+2}}
\ee
%and the rest of the proof is devoted to bounding 
%$E||\widehat{p}_{\widehat M}-p||^2_{L^2(q)}$.

Furthermore, let $f_{\bbeta}(\bx,y)$  be the joint 
distribution of $(\bX,Y)$ for a given $\bbeta$, i.e.
$f_{\bbeta}(\bx,y)=p(\bx)^y(1-p(\bx))^{1-y}q(\bx)$,
where $p(\bx)=\frac{\exp\{\bbeta^t \bx \}}{1+\exp\{\bbeta^t \bx \}}$.
Consider the square Hellinger distance $H^2(Bin(1,p_1),Bin(1,p_2))$
between two Bernoulli distributions with success probabilities
$p_1$ and $p_2$. It is easy to verify that
$H^2(Bin(1,p_1),Bin(1,p_2)) \geq \frac{(p_1-p_2)^2}{4(1-\delta)}
\geq \frac{(p_1-p_2)^2}{2}$.
Then for the square Hellinger distance $d^2_H(f_{\bbeta_1},f_{\bbeta_2})$ between $f_{\bbeta_1}$ and $f_{\bbeta_2}$ we have
\be \label{eq:h2lq}
d^2_H(f_{\bbeta_1},f_{\bbeta_2})=\int H^2(Bin(1,p_1(\bx)),
Bin(1,p_2(\bx)) q(\bx) d\bx \geq \frac{1}{2}||p_1-p_2||^2_{L^2(q)} 
\ee
and from (\ref{eq:L2q}) it is, therefore,  sufficient to bound the Hellinger risk 
$Ed^2_H(f_{\widehat \bbeta_{\widehat M}},f_{\bbeta})$.

We will show that the penalty (\ref{eq:penalty_random}) falls within a general
class of penalties considered in Yang and Barron (1998) and then apply their Theorem 1 to find an upper bound for $Ed^2_H(f_{\widehat \bbeta_{\widehat M}},f_{\bbeta})$. 

Using the standard inequality $\ln(1+t) \leq t$, under Assumption (A1) we have
\be \nonumber
\begin{split}
|\ln f_{\bbeta_2}(\bx,y)-\ln f_{\bbeta_1}(\bx,y)|&=\left|y \ln \frac{p_2(\bx)}{p_1(\bx)}+
(1-y)\ln\frac{1-p_2(\bx)}{1-p_1(\bx)}\right| \leq \max \left(\left|\ln \frac{p_2(\bx)}{p_1(\bx)}\right|,
\left|\ln \frac{1-p_2(\bx)}{1-p_1(\bx)}\right|\right) \\ 
& 
\leq \frac{1}{\delta}|p_2(\bx)-p_1(\bx)|
\end{split}
\ee
Define $\rho(f_{\bbeta_1},f_{\bbeta_2})=||\ln f_{\bbeta_2}-\ln f_{\bbeta_1}||_\infty$. Thus,
\be \label{eq:helinfty}
\rho(f_{\bbeta_1},f_{\bbeta_2}) \leq \frac{1}{\delta}||p_2-p_1||_\infty
\ee

For a given model $M$ consider the set of coefficients $\cb_M$ defined in Section \ref{sec:notation}.
One can easily verify that under Assumption (A1), for any $\bbeta_1, \bbeta_2 \in \cb_M$ and the corresponding $p_1(\bx), p_2(\bx)$ 
\be \label{eq:equivnorm}
\delta(1-\delta)\left|(\bbeta_2-\bbeta_1)^t \bx\right| \leq 
\left|p_2(\bx)-p_1(\bx)\right| \leq
\frac{1}{4}\left|(\bbeta_2-\bbeta_1)^t \bx\right|
\ee
for any $\bx \in \cX$.

In particular, (\ref{eq:equivnorm}) implies
\be \label{eq:l2norm}
\left||p_2(\bx)-p_1(\bx)\right||_{L_2(q)}  \geq \delta(1-\delta) \sqrt{(\bbeta_2-\bbeta_1)^t G (\bbeta_2-\bbeta_1)} \geq \delta (1-\delta) \sqrt{\lambda_{min}(G)}~
||\bbeta_2-\bbeta_1||_2,
\ee
where recall that $G=E(\bX \bX^t)$ and
$\lambda_{min}(G)>0$ is its minimal eigenvalue.

For each $\bbeta_0 \in \cb_M$ consider the corresponding
Hellinger ball $\cH_{f_{\bbeta_0},r}=\{f_{\bbeta}: d_H(f_{\bbeta},
f_{\bbeta_0}) \leq r, \;\bbeta \in \cb_M\}$.
From (\ref{eq:h2lq}) and (\ref{eq:l2norm}) it then follows that if $f_{\bbeta} \in \cH_{f_{\bbeta_0},r}$, the corresponding $\bbeta \in \cb_M$ lies in the
Euclidean ball  $\cb_{\bbeta_0,r'}=\{\bbeta \in \mathbb{R}^{|M|}:
||\bbeta-\bbeta_0||_2 \leq r'\}$ with $r'=\frac{\sqrt{2}r}{\delta(1-\delta)
\sqrt{\lambda_{min}(G)}}$. 
 
Furthermore, for any $||\bx||_2 \leq 1$, (\ref{eq:equivnorm}) and Cauchy–-Schwarz inequality imply
that $\left|p_2(\bx)-p_1(\bx)\right| \leq
\frac{1}{4} ||\bbeta_2-\bbeta_1||_2$ and, therefore, by (\ref{eq:helinfty})
\be \label{eq:inftynorm}
\rho(f_{\bbeta_1},f_{\bbeta_2}) \leq \frac{1}{4\delta}~ ||\bbeta_2-\bbeta_1||_2
\ee

Let $N(\cb_{\bbeta_0,r'},l_2, \epsilon)$ be the $\epsilon$-covering number of
$\cb_{\bbeta_0,r'}$ w.r.t. $l_2$-distance.  It is well-known that 
$N(\cb_{\bbeta_0,r'},l_2, \epsilon) \leq \left(1+\frac{2r'}{\epsilon}\right)^{|M|}
\leq \left(\frac{3r'}{\epsilon}\right)^{|M|}$ for any $\epsilon < r'$.

Thus, for the $\epsilon$-covering number
$N(\cH_{f_{\bbeta_0},r},\rho,\epsilon)$ of $\cH_{f_{\bbeta_0},r}$ w.r.t. the distance 
$\rho(f_{\bbeta_1},f_{\bbeta_2})$,
from (\ref{eq:inftynorm}) we have
$$
N(\cH_{f_{\bbeta_0},r},\rho,\epsilon) \leq N(\cb_{\bbeta_0,r'},l_2, 4\delta \epsilon) \leq 
\left(\frac{3 \sqrt{2}}{4\delta^2(1-\delta)\sqrt{\lambda_{\min}(G)}}~\frac{r}{\epsilon}\right)^{|M|}
$$ 
%where we used the boundedness assumption on $\lambda_{min}(W)$ from
%Section \ref{sec:random}. 

The considered family of sparse logistic regression models satisfies then Assumption 1
of Yang and Barron (1998) with $A_M=\frac{c}{\delta^2(1-\delta)\sqrt{\lambda_{\min}(G)}}$ for some $c>0$ and $m_M=|M|$. Apply 
now their Theorem 1 for a penalized maximum likelihood model
selection procedure (\ref{eq:binpen}) with a complexity
penalty $Pen(|M|)= C_1~m_M \ln A_M+
C_2 \cdot C_M \leq \tilde{C}_1 \ln \left(\frac{1}{\delta \lambda_{\min}(G)}\right)~|M|+C_2 |M| \ln\frac{de}{|M|}$, where 
$C_M=|M|\ln\frac{de}{|M|}$, and the exact positive constants $C_1$ and $C_2$ are given in the paper. Thus,
$$
Ed^2_H(f_{\widehat \bbeta_{\widehat M}},f_{\bbeta}) \leq \tilde{C}\ln \left(\frac{1}{\delta \lambda_{\min}(G)}\right)~ \frac{Pen(d_0)}{n}
$$
To complete the proof note that one can always find a constant
$C$ in the penalty (\ref{eq:penalty_random}) such that the resulting
$Pen(|M|)=C \ln \left(\frac{1}{\delta \lambda_{\min}(G)}\right)|M|\ln\frac{de}{|M|} \geq \tilde{C}_1 \left(\frac{1}{\delta \lambda_{\min}(G)}\right)~|M|+C_2 |M| \ln\frac{de}{|M|}$.

\end{document}